\documentclass{article}

\usepackage{amsmath}
\usepackage{amssymb}
\usepackage{mathrsfs}
\usepackage{amscd}
\usepackage{amsthm}
\usepackage{xypic}
\usepackage{latexsym}
\usepackage[dvips]{graphicx}
\newtheorem{thm}{Theorem}[section]
\newtheorem{prop}[thm]{Proposition}
\newtheorem{cor}[thm]{Corollary}
\newtheorem{lem}[thm]{Lemma}
\theoremstyle{definition}
\newtheorem{rmk}[thm]{Remark}
\newtheorem{exam}[thm]{Example}

\newcommand{\AFp}{\mathbb{A}_{\mathbb{F}_q}^1}
\newcommand{\PFp}{\mathbb{P}_{\mathbb{F}_q}^1}
\newcommand{\Ak}{\mathbb{A}_k^1}

\newcommand{\Z}{\mathbb{Z}}
\newcommand{\Fq}{{\mathbb{F}_q}}
\newcommand{\Fqr}{{\mathbb{F}_{q^r}}}
\newcommand{\Fp}{{\mathbb{F}_p}}
\newcommand{\Zoverpn}{{\mathbb{Z}/p^n}}
\newcommand{\WnFq}{{W_n(\mathbb{F}_q)}}
\newcommand{\Wnk}{{W_n(k)}}
\newcommand{\Ql}{\mathbb{Q}_\ell}
\newcommand{\Qlb}{{\overline{\mathbb{Q}}_\ell}}
\newcommand{\Gam}{\Gamma}
\newcommand{\sS}{\mathscr{S}}

\newcommand{\sX}{\mathscr{X}}

\newcommand{\COX}{{\mathcal{C}(\mathcal{O}_X)}}
\newcommand{\COY}{{\mathcal{C}(\mathcal{O}_Y)}}
\newcommand{\tGa}{\tilde{\Gamma}}
\newcommand{\ta}{\tilde{a}}
\newcommand{\tu}{\tilde{u}}
\newcommand{\tz}{\tilde{z}}
\newcommand{\tK}{\tilde{K}}

\newcommand{\cA}{\mathcal{A}}
\newcommand{\cC}{\mathcal{C}}

\newcommand{\cF}{\mathcal{F}}
\newcommand{\cG}{\mathcal{G}}
\newcommand{\cI}{\mathcal{I}}
\newcommand{\cK}{\mathcal{K}}

\newcommand{\cO}{\mathcal{O}}
\newcommand{\cX}{\mathcal{X}}

\newcommand{\cGb}{{\overline{\mathcal{G}}}}
\newcommand{\cGo}{\mathcal{G}_0}
\newcommand{\cGob}{{\overline{\mathcal{G}}_0}}
\newcommand{\id}{\mathrm{id}}
\newcommand{\Xb}{\overline{X}}
\newcommand{\Yb}{\overline{Y}}
\newcommand{\fb}{\overline{f}}
\newcommand{\gb}{\overline{g}}
\newcommand{\Xob}{\overline{X}_0}
\newcommand{\Gamb}{{\overline{\Gam}}}
\newcommand{\Gamob}{{\overline{\Gam}_0}}
\newcommand{\Fr}{\mathrm{Fr}}
\newcommand{\red}{\mathrm{red}}
\newcommand{\pr}{\mathrm{pr}}
\newcommand{\adj}{\mathrm{adj}}
\newcommand{\vP}{\varPhi}
\newcommand{\vp}{\varphi}
\newcommand{\al}{\alpha}
\newcommand{\bt}{\beta}
\newcommand{\sg}{\sigma}
\newcommand{\La}{\Lambda}
\newcommand{\ab}{\overline{a}}
\newcommand{\aib}{\overline{a}_1}
\newcommand{\azb}{\overline{a}_2}
\newcommand{\aoib}{\overline{a}_{01}}
\newcommand{\aozb}{\overline{a}_{02}}
\newcommand{\Gob}{{\overline{G}_0}}
\newcommand{\Gb}{\overline{G}}

\newcommand{\ub}{\overline{u}}

\newcommand{\jGa}{{j_\Gamma}}

\newcommand{\Spec}{\mathop{\mathrm{Spec}}\nolimits}
\newcommand{\Hom}{\mathop{\mathrm{Hom}}\nolimits}

\newcommand{\Ker}{\mathop{\mathrm{Ker}}\nolimits}
\newcommand{\Coker}{\mathop{\mathrm{Coker}}\nolimits}
\newcommand{\RGa}{\mathop{R\Gam}\nolimits}
\newcommand{\RGac}{\mathop{R\Gam _c}\nolimits}
\newcommand{\Tr}{\mathop{\mathrm{Tr}}\nolimits}
\newcommand{\Fix}{\mathop{\mathrm{Fix}}\nolimits}

\newcommand{\naiveloc}{\mathop{\mathrm{naive\mathchar`- loc}}\nolimits}
\newcommand{\Res}{\mathop{\mathrm{Res}}\nolimits}
\newcommand{\Dperf}{\mathop{D_\mathrm{perf}}\nolimits}
\newcommand{\Dctfb}{\mathop{D_\mathrm{ctf}^b}\nolimits}
\newcommand{\ord}{\mathop{\mathrm{ord}}\nolimits}

\newcommand{\cln}{\colon}

\newcommand{\isomto}{\xrightarrow{\sim}}

\newcommand{\immto}{\hookrightarrow}
\newcommand{\smin}{\smallsetminus}
\newcommand{\ts}{\mathop{\otimes}}

\newcommand{\tbl}{\textbullet}

\def\maru#1{{\ooalign{\hfil
  \ifnum#1>999 \resizebox{.25\width}{\height}{#1}\else%
  \ifnum#1>99 \resizebox{.33\width}{\height}{#1}\else%
  \ifnum#1>9 \resizebox{.5\width}{\height}{#1}\else #1%
  \fi\fi\fi%
\/\hfil\crcr%
\raise.167ex\hbox{\mathhexbox20D}}}}

\title{Deligne's conjecture on the Lefschetz trace formula for $p^n$-torsion \'etale cohomology}
\author{Megumi Takata}
\date{\today}

\begin{document}

\maketitle

\begin{abstract}
Deligne's conjecture is the Lefschetz trace formula for correspondences defined over a finite field.
In this paper, we prove an analogous statement of Deligne's conjecture
with respect to $p^n$-torsion \'etale cohomology under certain conditions,
where $p$ is the characteristic of the base field.
\end{abstract}

\tableofcontents

\begin{section}{Introduction}
In abstract algebraic geometry, we have the Lefschetz trace formula for correspondences of proper smooth schemes,
which is a generalization of the Lefschetz fixed-point formula.
The trace formula does not hold when schemes are not proper or smooth.
However, in the positive characteristic case, we have a nice Lefschetz trace formula
for correspondences twisted by a sufficiently large power of the Frobenius endomorphism.
It was conjectured by Deligne, and proved by Fujiwara \cite{Fu} under the most general situation.
This trace formula is an important tool in arithmetic geometry.
For example, it appears in 
the proof of the local Langlands correspondence for $\mathit{GL}_n$ over a $p$-adic field by Harris-Taylor \cite{HT}, and
the Langlands correspondence for $\mathit{GL}_n$ over a function field by Lafforgue \cite{Laf}.

We first recall Deligne's conjecture.
We fix a prime number $p$ and take a prime number $\ell \neq p$.
Let $q$ be a power of $q$, $\Fq$ the finite field with $q$ elements and $k$ an algebraic closure of $\Fq$.
Let $X$ and $\Gam$ be separated $k$-schemes of finite type
and $a\cln \Gam \to X\times X$ a morphism of $k$-schemes.
We assume that $X$, $\Gam$ and $a$ are defined over $\Fq$.
We put $a_1=\pr_1\circ a$ and $a_2=\pr_2\circ a$,
where $\pr_1$ (resp. $\pr_2$) is the first (resp. second) projection of $X\times_k X$.
We denote by $\Fr_X$ the $q$-th power Frobenius endomorphism on $X$.
We write $a^{(m)}$ for the correspondence such that $a^{(m)}_1 = \Fr_X^m\circ a_1$ and $a^{(m)}_2 = a_2$.
We write $D_c^b(X,\Qlb)$ for the derived category of bounded complexes of $\Qlb$-sheaves with constructible cohomology
(for more detail of this category, see for example \cite{Ek}).
\begin{thm}[Deligne's conjecture]\label{thm1}
We assume that $a_1$ is proper and $a_2$ is quasi-finite.
Then there exists an integer $N$ depending only on $X$, $\Gam$ and $a$ such that,
for any integer $m \geq N$, any object $K\in D_c^b(X,\Qlb)$ and any morphism $u\in\Hom(a_1^{(m)*}K,a_2^!K)$,
the equation
\[
	\Tr(u_! \mid \RGac(X,K)) = \sum_{z\in\Fix a^{(m)}} \naiveloc_z(u)
\]
holds.
\end{thm}
In the above equation, $u_!$ is defined by the composition
\[
	\RGac(X,K) \xrightarrow{\adj} \RGac(\Gam,a_1^{(m)*}K) \xrightarrow{u}
		\RGac(\Gam,a_2^!K) \xrightarrow{\adj} \RGac(X,K)
\]
and $\naiveloc_z(u)$ is defined by the trace of the composition
\[
	K_w = (a_1^{(m)*}K)_z \xrightarrow{u_z} (a_2^!K)_z
		\immto \bigoplus_{z'\in a_2^{-1}(w)}(a_2^!K)_{z'} = (a_{2!}a_2^!K)_w
		\xrightarrow{\adj} K_w,
\]
where we put $w = a_1^{(m)}(z) = a_2(z)$.

In this paper, we will show an analogous statement to Theorem \ref{thm1}
for \'etale cohomology of $p^n$-torsion sheaves.
Let $X_0$ and $\Gam_0$ be separated $\Fq$-schemes of finite type
and $a_0\cln \Gam_0 \to X_0\times_{\Fq}X_0$ an $\Fq$-morphism.
We write $\vP_{X_0}$ for the $p$-th power map on the structure sheaf of $X_0$.
Let $G_0$ be a constructible $\Z/p^n$-sheaf on $X_0$.

\begin{thm}\label{thm12}
We assume $a_1$ is proper and $a_2$ is \'etale.
Then there exists an integer $N$ depending on $X_0$, $\Gam_0$, $a_0$ and $G_0$
such that, for any integer $m\geq N$ and any $u_0 \in \Hom(a_{01}^{(m)*}G_0,a_{02}^*G_0)$, we have
\[
	\Tr (u_! \mid \RGac (X, G)) = \sum_{z \in \Fix a^{(m)}} \Tr(u_z \mid G_z)
\]
under one of the following conditions:
\begin{itemize}
\item[{\rm(i)}]
	(Corollary \ref{cor32})
	\begin{itemize}
	\item[\tbl] $n=1$,
	\item[\tbl] $a_2 = \id_X$ and $a_1$ is an automorphism of finite order,
	\item[\tbl] $\ord u = \ord a_1$.
	\end{itemize}
\item[{\rm(ii)}]
	(Theorem \ref{thm41})
	\begin{itemize}
	\item[\tbl] $n=1$,
	\item[\tbl] $a_0$ is a restriction of a correspondence $\ab_0\cln \Gamob \to \Xob \times_{\Fq} \Xob$ such that
			$\Xb$ and $\Gamb$ are proper smooth $k$-schemes,
			$X_0$ (resp. $\Gam_0$) is an open $\Fq$-subscheme of $\Xob$ (resp. $\Gamob$),
			$\ab$ is a closed immersion, $\ab_2$ is \'etale and $\Xb \smin X$ is a Cartier divisor,
	\item[\tbl] $G_0$ is the pull-back of a smooth constructible $\Fp$-sheaf on $\Xob$.
	\end{itemize}
\item[{\rm(iii)}]
	(Theorem \ref{thm51})
	\begin{itemize}
	\item[\tbl] $\ab$ is a closed immersion,
	\item[\tbl] $\Xb$ and $\Gamb$ are proper smooth $k$-schemes,
	\item[\tbl] there exists a lift $(\sX_0,\tGa_0,\ta_0,\vP_{\sX_0})$ of $(X_0,\Gam_0,a_0,\vP_{X_0})$ to $\WnFq$,
	\item[\tbl] $H^i(X,G)$ (resp. $H^i(\sX,G\ts_{\Z/{p^n}}\cO_\sX)$) is free over $\Z/p^n$ (resp. $\Wnk$) for each $i$.
	\end{itemize}
\end{itemize}
In the case {\rm(i)} and {\rm(ii)}, we can take $N=1$.
\end{thm}

The first case for $a_1=\id$ is already known (\cite[Fonction $L$ mod.\ $\ell^n$]{SGA4h}).
By using the method in \cite{DL}, we extend the result to the case where $a_1$ is an automorphism of finite order.

For the second and last cases, we generalize the proof of the trace formula for the Frobenius correspondence.
The outline of the proof is as follows.
First, by tensoring with a structure sheaf,
we transform the left-hand side of the formula in Theorem \ref{thm12}
into a trace of an endomorphism of cohomology groups with respect to coherent sheaves.
After that, we apply certain trace formula called ``Woods Hole formula'' to this term.
Then the term transforms into a sum of traces of stalks of a homomorphism of coherent sheaves at fixed points.
Finally, by a calculation of the stalks, we have the result.

We devote Section\ \ref{sect2} to recalling the terminology of cohomological operations and correspondences.
In Section\ \ref{sect3} we prove Theorem \ref{thm12} in the case (i).
In Section\ \ref{sect4} we prove the Woods Hole formula which is a Lefschetz trace formula for coherent sheaves.
In Section\ \ref{sect5} we recall some linear algebra to prove the cases (ii) and (iii).
Under the preliminaries above, we prove Theorem \ref{thm12} in the case (ii) (resp. the case (iii))
in Section\ \ref{sect6} (resp. Section\ \ref{sect7}).
Finally, we show an example such that the trace formula holds, an application to elliptic curves, and a counterexample for $p$-adic coefficients in Section\ \ref{sect8}.
\end{section}

\section*{Notation}
\begin{itemize}
\item
	In this paper, we fix a prime number $p$.
	We denote by $\Fq$ the finite field of $q$ elements where $q$ is a power of $p$
	and fix an algebraic closure $k$ of $\Fq$.
	For any field $K$ of characteristic $p$, we write $W_n(K)$ for the ring of Witt vectors of length $n$ over $K$. 
\item
	For any $\WnFq$-object $\cX_0$, we write $\cX = \cX_0 \ts_{\WnFq} \Wnk$ for the base change of $\cX_0$ by $\WnFq \to \Wnk$.
\item
	Let $X_0$ be a $\Fq$-scheme.
	We write $\Fr_{X_0} \cln X_0 \to X_0$ for the Frobenius endomorphism on $X_0$,
	which is the pair of the identity on the underlying space of $X_0$ and the $q$-th power map on the structure sheaf of $X_0$.
	We write $\Fr_X$ for the base change of $\Fr_{X_0}$ by $\Fq \to k$.	
\item
	Let $f_0 \cln X_0 \to Y_0$ and $a_0 \cln \Gam_0 \to X_0\times_{\Fq} X_0$ be morphisms of $\Fq$-schemes.
	We put $f_0^{(m)} = \Fr_{Y_0}^m \circ f_0 = f_0 \circ \Fr_{X_0}^m$ and $a_0^{(m)} = (a_{01}^{(m)},a_{02})$,
	where we write $a_{01}$ (resp. $a_{02}$) for the composition of $a_0$ and the first (resp. second) projection of $X_0 \times_{\Fq} X_0$.
	We also define $f^{(m)}$ and $a^{(m)}$ in the same way.
\item
	Let $A$ be a ring.
	We denote by $D(A)$ the derived category of the category of $A$-modules.
	We denote by $\Dperf (A)$ the full subcategory of $D(A)$ which consists of perfect complexes,
	that is, complexes which are quasi-isomorphic to bounded complexes of projective $A$-modules of finite type.
\item
	Let $X$ be a scheme and $\cA$ a sheaf of rings on $X$.
	We denote by $D(X,\cA)$ the derived category of the category of $\cA$-module sheaves on $X$.
\item
	We denote by $\ts_\cA^L \cln D^-(X,\cA) \times D^-(X,\cA) \to D^-(X,\cA)$ the derived tensor product functor.
	If $\cF$ is a flat $\cA$-module, we write $\ts_\cA\cF$ instead of $\ts_\cA^L\cF$.
\item
	We denote by $\Dperf (X,\cA)$ the full subcategory of $D(X,\cA)$ which consists of perfect complexes,
	that is, complexes which are locally quasi-isomorphic to bounded complexes of locally free $\cA$-modules of finite type.
\item
	If $\cA$ is a constant sheaf $\La$ whose value is a noetherian ring,
	we denote by $\Dctfb(X,\La)$ the category of complexes of finite tor-dimension with constructible cohomology
	(see for example \cite[Rapport]{SGA4h}).
\item
	If $\cA = \cO_X$, we simply write $D(\cO_X)$ (resp. $\Dperf(\cO_X)$) for $D(X,\cO_X)$ (resp. $\Dperf(X,\cO_X)$).
\end{itemize}

\subsection*{Acknowledgment}
It is a great pleasure for the author to thank Professor Yoichi Mieda. 
He suggested many things about the problem,
checked the preliminary manuscripts carefully,
and pointed out some mistakes in it.

The author is obliged to professor Yuichiro Taguchi for
inviting the author warmly into his office
and giving valuable comments for this paper.

\begin{section}{Cohomological correspondences}\label{sect2}

In the section, we recall notation of cohomological operations and correspondences.

\begin{subsection}{Cohomological operations}
\begin{subsubsection}{Sheaves of $\La$-modules}\label{sss211}
Let $\La$ be a noetherian torsion ring and $S$ a spectrum of an artinian local ring.
Let $f\cln X\to Y$ be a morphism of separated schemes of finite type over $S$.

We denote by $\cC(X,\La)$ (resp. $\cC(Y,\La)$) the category of \'etale sheaves of $\La$-modules on $X$ (resp. $Y$). 
We denote the direct (resp. inverse) image functor associated to $f$
by $f_*\cln\cC(X,\La)\to\cC(Y,\La)$ (resp. by $f^*\cln\cC(Y,\La)\to\cC(X,\La)$).
Since $f^*$ is an exact functor, it induces a functor from $D(Y,\La)$ to $D(X,\La)$ which we also denote by $f^*$.
Since $f_*$ is a left exact functor and $f$ is finite-dimensional,
it induces a right derived functor $Rf_* \cln D(X,\La) \to D(Y,\La)$.
We often denote $Rf_*$ by $f_*$ for simplicity.
When $f$ is the structure morphism, we denote $Rf_*$ by $\RGa(X,\bullet)$.

For any complex $K\in D(X,\La)$, we write $\adj \cln K\to f_*f^*K$ for the adjunction map.

Let $j\cln X\immto\Xb$ be an open immersion.
We write $j_!\cln \cC(X,\La) \to \cC(\Xb,\La)$ for the extension-by-zero functor.
Since $j_!$ is exact, it induces a functor from $D(X,\La)$ to $D(\Xb,\La)$ which we also denote by $j_!$.

By Nagata's compactification theorem, we can decompose $f\cln X \to Y$ as follows:
\[
\xymatrix{
	X \ar@{^{(}->}[r]^-{j} \ar[rd]_-{f} & \Xb \ar[d]^-{\fb} \\
	& Y,
}
\]
where $\fb$ is a proper morphism and $j$ is an open immersion.
We define the direct image functor with proper support $f_!\cln D(X,\La) \to D(Y,\La)$ by the composite of $j_!$ and $\fb_*$.
When $f$ is the structure morphism of $X$,
we denote $f_!$ by $\RGac(X,\bullet)$. 

The functor $f^*$ sends $\Dctfb(Y,\La)$ to $\Dctfb(X,\La)$.
On the other hand, by \cite[Exp.\ XIV,\ Thm.\ 1.1 and Exp.\ XVII,\ Thm.\ 5.2.10]{SGA4},
the functor $f_!$ sends $\Dctfb(X,\La)$ to $\Dctfb(Y,\La)$.
In particular, if we take $K$ in $\Dctfb(X,\La)$ then $\RGac(X,K)$ belongs to $\Dperf(\La)$.

We assume that $g\cln Y \to X$ is finite and \'etale.
Then, for any abelian \'etale sheaf $G$ on $X$, we can define a morphism
\[
	\Tr_{g,G}\cln g_*g^*G \to G
\]
such that it is compatible with \'etale base change and if $Y$ is a disjoint sum of $d$ copies of $X$
then $\Tr\cln g_*g^*G=G^d\to G$ is the morphism $(x_i)\mapsto\sum_ix_i$,
where $x_1,\ldots,x_d$ are sections of $G$
(see for example \cite[Exp.\ VII, Sect.\ 5]{SGA4}).
Since the functor $g_*g^*$ is exact, we can extend it to a morphism of complexes $\Tr_{g,K}\cln g_*g^*K \to K$.
We note that the morphism $\Tr_{g,K}$ is functorial with respect to $K$.

We remove the assumption that $g$ is finite.
We write $g=\gb\circ j$ for a decomposition of $g$,
where $j\cln Y \immto \Yb$ is an open immersion and $\gb\cln\Yb\to X$ is a proper morphism.
Then $\gb$ is finite \'etale morphism, hence we can define a morphism
\[
	\Tr_{g,K} \cln g_!g^*K \to K
\]
by the composition
\[
	g_!g^*K = \gb_*j_!j^*\gb^*K \xrightarrow{\adj} \gb_*\gb^*K \xrightarrow{\Tr_{\gb,K}} K.
\]

\end{subsubsection}

\begin{subsubsection}{Sheaves of $\cO_X$-modules}
For a scheme $X$, we denote by $\COX$ the category of $\cO_X$-sheaves.
Let $f\cln X\to Y$ be a morphism of separated schemes of finite type over $S$.

We can naturally define the direct image functor
$\cln\COX\to\COY$
and its right derived functor from $D(\cO_X)$ to $D(\cO_Y)$, which we denote by $f_*$.

We define the inverse image functor as follows:
\[
	f^*\cln \COX \to \COY \cln \cG \mapsto f^{-1}\cG \ts_{f^{-1}\cO_Y}\cO_X.
\]
Since $f^*$ is right exact, we can define its left derived functor
$D^-(\cO_Y)\to D^-(\cO_X)$,
which we also denote by $f^*$.
When we treat \'etale sheaves and $\cO_Y$-modules at the same time,
we write $f^{-1}$ for the inverse image functor of \'etale sheaves to avoid confusion.

For any complex $\cK\in D^-(\cO_X)$, we write $\adj\cln \cK\to f_*f^*\cK$ for the adjunction map
which coincides with the composite of the adjunction map $\cK \to f_*f^{-1}\cK$ defined in (\ref{sss211})
and the canonical morphism $f_*f^{-1}\cK \to f_*f^*\cK$.

If $X$ and $Y$ are smooth over $S$ and $f$ is proper,
then by \cite[Exp.\ III, Cor.\ 4.8.1]{SGA6},
$f_*$ sends $\Dperf(\cO_X)$ to $\Dperf(\cO_Y)$.

Let $g\cln Y\to X$ be a finite \'etale morphism.
Then we have the trace map
\[
\Tr_{g,\cO_X}\cln g_*\cO_Y \to \cO_X.
\]
Since $g$ is \'etale and finite, for any complex $\cK$ in $D(\cO_X)$,
we have a canonical isomorphism $\cK\ts_{\cO_X} g_*\cO_Y \isomto g_*g^*\cK$.
Thus, by tensoring $\Tr_{g,\cO_X}$ with $\cK$ and composing the canonical isomorphism, we obtain
\[
	\Tr_{g,\cK} \cln g_*g^*\cK \to \cK.
\]
This morphism is functorial with respect to $\cK$.
\end{subsubsection}

\begin{subsubsection}{Relation between $\Zoverpn$-modules and $\cO_X$-modules}
We suppose that $X$ and $Y$ are smooth over $S$.
Let $f\cln Y \to X$ and $g\cln Y \to X$ be morphisms of schemes over $S$.
We assume that $f$ is proper and $g$ is finite and \'etale.

Let $K$ be a complex in $D(X,\Zoverpn)$ and $\cK$ a complex in $D^-(\cO_X)$.
Let $u\cln K \to \cK$ be a morphism of complexes.
We define the morphism
\[
	f_*f^*u\cln f_*f^{-1}K \to f_*f^*\cK
\]
by the composite of
\[
	f_*f^{-1}u\cln f_*f^{-1}K \to f_*f^{-1}\cK
\]
and the canonical map
\[
	f_*f^{-1}\cK \to f_*f^*\cK.
\]
We also define $g_*g^*u$ by the same way.
\begin{prop}\label{prop21}
The diagrams
\[
\xymatrix{
	K \ar[r]^-{\adj} \ar[d]_{u} & f_*f^{-1}K \ar[d]^{f_*f^*u} \\
	\cK \ar[r]^-{\adj} & f_*f^*\cK
}
\]
and
\[
\xymatrix{
	g_*g^{-1}K \ar[r]^-{\Tr} \ar[d]_{g_*g^*u} & K \ar[d]^{u}\\
	g_*g^*\cK \ar[r]^-{\Tr} & \cK
}
\]
are commutative.
\end{prop}
\begin{proof}
The commutativity of the first diagram follows from the definition of $f_*f^*u$
and $\adj \cln \cK \to f_*f^*\cK$.
For the commutativity of the second, it suffices to show the diagram
\[
\xymatrix{
	g_*g^{-1}\cK \ar[r]^-{\Tr} \ar[d] & \cK \\
	g_*g^*\cK \ar[ru]_-{\Tr} & 
}
\]
is commutative.
This follows from the commutativity of the diagrams
\[
\xymatrix{
	\cK\ts_\Zoverpn g_*\Zoverpn \ar[r]^-{\sim} \ar[d] & g_*g^{-1}\cK \ar[d] \\
	\cK\ts_{\cO_X} g_*\cO_Y \ar[r]^-{\sim} & g_*g^*\cK
}
\]
and
\[
\xymatrix{
	g_*\Zoverpn \ar[r]^-{\Tr} \ar[d] & \Zoverpn \ar[d] \\
	g_*\cO_Y \ar[r]^-{\Tr} & \cO_X.
}
\]
\end{proof}
\end{subsubsection}
\end{subsection}

\begin{subsection}{Cohomological correspondences}
\begin{subsubsection}{Correspondences}
Let $X$ and $\Gam$ be separated schemes of finite type over $S$.
Let $a\cln \Gam \to X\times_S X$ be a $S$-morphism of schemes.
Such a morphism is called a correspondence.
For a correspondence $a$, we denote by $\Fix a$ the fiber product $\Gam\times_{X\times_S X}X$,
that is, $\Fix a$ is given by the cartesian diagram
\[
\xymatrix{
	\Fix a \ar[r] \ar[d] & X \ar[d]^{\Delta_{X/S}} \\
	\Gam \ar[r]^-{a} & X\times_S X,
}
\]
where the morphism $\Delta_{X/S}$ is the diagonal map.
We write $a_1$ (resp. $a_2$)
for the composite of $a$ and the first (resp. second) projection of $X\times_S X$.

In this section, we assume that $a_1$ is proper and $a_2$ is \'etale.

Let $K$ be a complex in $D(X,\La)$.
We define $u_!$ by the composition
\[
\xymatrixcolsep{2.5pc}\xymatrix{
	\RGac(X,K) \ar[r]^{\RGac(\adj)} & \RGac(\Gam,a_1^*K) \ar[r]^{\RGac(u)}
		& \RGac(\Gam,a_2^*K) \ar[r]^{\RGac(\Tr)} & \RGac(X,K) \lefteqn{.}
}
\]
When the schemes $X$ and $\Gam$ are proper over $S$, we often denote $u_!$ by $u_*$.

For any geometric point $z$ of $\Fix a$,
we put $K_z = (a_1^*K)_z = (a_2^*K)_z$.
Then the stalk
\[
	u_z\cln K_z = (a_1^*K)_z \to (a_2^*K)_z = K_z
\]
is an endomorphism on $K_z$.
If $K$ is a complex in $\Dctfb(X,\La)$ then $K_z$ lies in $\Dperf(\La)$.
Hence we can consider the trace $\Tr(u_z \mid K_z)$, which is an element of $\La$.

Now, we assume that $X$ and $\Gam$ are proper over $S$.
In this case, $a_2$ is finite.

For any complex $\cK$ in $D^-(\cO_X)$ and $u\in\Hom(a_1^*\cK,a_2^*\cK)$,
we can also define the endomorphism $u_*$ on $\RGa(X,\cK)$ in the same way.

Let $\cK$ be a complex in $D^-(\cO_X)$ and $u$ a morphism in $\Hom(a_1^*\cK,a_2^*\cK)$.
For any connected component $\bt$ of $\Fix a$,
we put $\cK_\bt = i_\bt^*a_1^*\cK = i_\bt^*a_2^*\cK$ and
\[
	u_\bt=i_\bt^*u \cln \cK_\bt = i_\bt^*a_1^*\cK \to i_\bt^*a_2^*\cK = \cK_\bt,
\]
where we write $i_\bt$ for the closed immersion $\bt \immto \Gam$.

We assume that $X$, $\Gam$ and $\Fix a$ are smooth over $S$.
If $\cK$ is a complex in $\Dperf(\cO_X)$,
by \cite[Exp.\ III,\ Cor.\ 4.5.1, Rem.\ 4.6.2 and Exp.\ VII,\ Prop.\ 1.9]{SGA6},
we have $\cK_\bt = i_\bt^* a_2^* \cK_\bt \in \Dperf(\cO_\bt)$.
So we can consider the trace $\Tr(u_\bt \mid \cK_\bt)$, which is an element of $\Gam(\bt,\cO_\bt)$.
\end{subsubsection}

\begin{subsubsection}{Push-forward for correspondences}
Let $X$ and $\Gam$ be separated $S$-schemes of finite type.
We consider open immersions $j_X\cln X\immto\Xb$ and $j_\Gam\cln \Gam\immto\Gamb$,
and correspondences $a\cln \Gam \to X$ and $\ab\cln \Gamb \to \Xb$ such that the diagram
\[
\xymatrix{
	\Gam \ar[r]^-{a} \ar@{^(->}[d]_{j_\Gam} & X\times_S X \ar@{^(->}[d]^{j_X\times j_X} \\
	\Gamb \ar[r]^-{\ab} & \Xb\times_S \Xb
}
\]
is commutative.
We assume that $a_1$ and $\aib$ are proper.
Let $K$ be a complex in $D(X,\La)$.
Then we define the morphism
\[
	BC_2 \cln \jGa_!a_2^*K \to \azb^*j_{X!}K
\]
by the composition
\[
	\jGa_!a_2^*K \xrightarrow{\jGa_!a_2^*(\adj)} \jGa_!a_2^*j_X^*j_{X!}K = \jGa_!\jGa^*\azb^*j_{X!}K \xrightarrow{\adj} \azb^*j_{X!}K.
\]
By the properness of $a_1$ and $\aib$, we also have the morphism
\[
	j_{X!}K \xrightarrow{j_{X!}(\adj)} j_{X!}a_{1*}a_1^*K = j_{X!}a_{1!}a_1^*K = {\aib}_!\jGa_!a_1^*K = {\aib}_*\jGa_!a_1^*K.
\]
Thus, by adjointness, we can define the morphism
\[
	BC_1 \cln \aib^*j_{X!}K \to \jGa_!a_1^*K.
\]
For any $u\in\Hom(a_1^*K,a_2^*K)$, we define $j_!u\in\Hom(\aib^*j_{X!}K,\azb^*j_{X!}K)$ by the composition
\[
\xymatrix{
	 \aib^*j_{X!}K \ar[r]^{BC_1} & \jGa_!a_1^*K \ar[r]^{\jGa_!u} & \jGa_!a_2^*K \ar[r]^{BC_2} & \azb^*j_{X!}K.
}
\]
\end{subsubsection}

\begin{subsubsection}{Frobenius correspondences}
Let $X_0$ be a $\Fq$-scheme and $G_0$ a sheaf on $X_0$.
Then we have an canonical isomorphism
\[
	\Fr_{G_0} \cln \Fr_{X_0}^*G_0 \isomto G_0.
\]
We call it the Frobenius correspondence of $G_0$
(see \cite[Rapport, 1.2]{SGA4h}).
The Frobenius correspondence is functorial, that is,
for any morphism $f_0\cln X_0 \to Y_0$ of $\Fq$-schemes and any sheaf $G_0$ on $Y_0$,
the pull-back $f_0^*\Fr_{G_0}$ is the Frobenius correspondence of $f_0^*G_0$.
For any integer $m\geq 1$, we put
\[
\Fr_{G_0}^m = \Fr_{G_0}\circ\Fr_{X_0}^*\Fr_{G_0}\circ \cdots \circ \Fr_{X_0}^{(m-1)*}\Fr_{G_0}.
\]

We define an isomorphism $\Fr_G \cln \Fr_X^*G \isomto G$ as the pull-back of $\Fr_{G_0}$ by $X \to X_0$
and also call it the Frobenius correspondence of $G$.
Note that $\Fr_G$ is functorial as well as $\Fr_{G_0}$.
We define $\Fr_G^m$ similarly to $\Fr_{G_0}^m$.

By abuse of notation, we often write $\Fr$ for $\Fr_{G_0}$ or $\Fr_G$.

Since $\Fr$ is isomorphic, we can replace the morphism $u_0$ in Theorem \ref{thm12}
by the morphism $u_0 \circ \Fr^m$ where $u_0 \in \Hom(a_{01}^*G_0,a_{02}^*G_0)$.
When we will state our result later, we use the replaced notation.
\end{subsubsection}

\end{subsection}
\end{section}

\begin{section}{The case of an automorphism of finite order}\label{sect3}
First, we recall the Lefschetz trace formula for the Frobenius correspondence
which is proved in \cite[Fonction $L$ mod.\ $\ell$]{SGA4h}.
After that, we extend this formula to the case of an automorphism of finite order
by using the method in the proof of \cite[Proposition 3.3]{DL}.

\begin{thm}[\rm{\cite[Fonction $L$ mod.\ $\ell$, Thm. 4.1]{SGA4h}}]\label{thm211}
Let $X_0$ be a separated scheme of finite type over $\Fq$,
and $A$ a noetherian reduced ring of characteristic $p$.
Then for any object $K_0$ in $\Dctfb(X_0,A)$ and any integer $m \geq 1$, we have
\[
	\Tr((\Fr_K^m)_! \mid \RGac(X,K)) = \sum_{x\in \Fix(\Fr_X^m\times\id)} \Tr((\Fr_K^m)_x \mid K_x).
\]
\end{thm}

\begin{cor}\label{cor32}
Let $X_0$ be a separated scheme of finite type over $\Fq$
and $g_0$ an automorphism of order $r$ on $X_0$.
Then for any constructible $\Fp$-module $G_0$ on $X_0$,
any integer $m \geq 1$ and any morphism $u_0 \in \Hom(g_0^*G_0, G_0)$ of order $r$, we have
\[
	\Tr((u\circ \Fr^m)_! \mid \RGac(X,G)) = \sum_{x\in \Fix(g^{(m)}\times\id)} \Tr((u\circ\Fr^m)_x \mid G_x).
\]
\end{cor}
\begin{proof}
By changing the field of definition of $X$, we may assume that $m=1$.

Since $X_0$ is separated of finite type over $\Fq$ and $g_0$ is of finite order,
we can construct a finite partition ${X_{0i}}$ of $X_0$ such that
$X_{0i}$ is locally closed in $X_0$, affine, $g_0$-stable for each $i$, and $G_0$ is smooth on $X_{0i}$.
Then we have
\[
	\Tr((u\circ\Fr)_! \mid \RGac(X,G)) = \sum_i\Tr((u\circ \Fr)_! \mid \RGac(X_i,G_i))
\]
and
\[
	\sum_{x\in \Fix(g^{(1)}\times\id)} \Tr((u\circ\Fr)_x \mid G_x)
		= \sum_i\sum_{x\in\Fix(g_i^{(1)}\times\id)} \Tr((u\circ\Fr)_x \mid G_x),
\]
where $g_i$ (resp. $G_i$) is the restriction of $g$ (resp. $G$) to $X_i$.
Hence we may assume that $X_0$ is affine and $G_0$ is smooth.

We put $X_1=X_0\ts_\Fq\Fqr$, $G_1=G_0|_{X_1}$ and $u_1=u_0|_{X_1}$.
We denote by $\sg$ the isomorphism on the scheme $\Spec \Fqr$ associated to the $q$-th power map on $\Fqr$.
Since $X_1$ is affine, by using $g_0\ts\sg^{-1}$ as a descent datum,
we can construct the scheme $X_0'$ over $\Fq$ such that
$X_0'\ts_\Fq\Fqr \simeq X_1$ and $g^{(1)}$ is the relative Frobenius endomorphism of $X$ with respect to $X_0$.
Since $G_1$ is smooth, we can also construct the sheaf $G_0'$ on $X_0'$ such that $G_0'|_{X_1} \simeq G_1$
and $u\circ\Fr$ is equal to the Frobenius correspondence on $G$ with respect to $G_0'$.
Therefore by Theorem \ref{thm211} we obtain
\[
	\Tr((u\circ\Fr)_! \mid \RGac(X,G)) = \sum_{x\in \Fix(g^{(1)}\times\id)} \Tr((u\circ\Fr)_x \mid G_x).
\]
\end{proof}
\end{section}

\begin{section}{Woods Hole formula}\label{sect4}

In this section we prove a trace formula for coherent sheaves called ``Woods Hole formula''.
We use this theorem in Section \ref{sect6} and \ref{sect7}.

\begin{thm}\label{thm31}
Let $S$ be a spectrum of an artinian local ring.
Let $X$ and $\Gam$ be proper smooth schemes over $S$.
We denote by $\pi_X$ the structure morphism of $X$.
Let $a\cln\Gam\immto X\times_S X$ be a closed immersion over $S$. 
Assume that $a_2$ is \'etale and
the homomorphism 
$da_1\cln a_1^*\Omega_{X/S}\to\Omega_{\Gam/S}$
induced by $a_1$ is zero.
Then for any object $\cK$ in $\Dperf(\cO_X)$ and $u\in\Hom(a_1^*\cK,a_2^*\cK)$, we have
\[
	\Tr(u_* \mid \RGa(X,\cK)) = \sum_{\bt\in\pi_0(\Fix a)}\Tr_{\bt/S}(\Tr(u_\bt\mid \cK_\bt)).
\]
\end{thm}

\begin{rmk}
\begin{itemize}
\item
	Since $X$ and $\Gam$ are smooth over $S$ and $da_1 = 0$,
	$\Gam$ meets transversally to the diagonal of $X$, that is, $\Fix a$ is \'etale over $S$
	\cite[IV Cor.\ 17.13.6]{EGA}.
	Hence $\Fix a$ is a finite direct sum of the spectrums of local artinian rings which are \'etale over $S$.
\item
	Assume that $S= \Spec\Wnk$ and $X$ is defined over $\WnFq$.
	We suppose that there exists a lift $\Fr_X\cln X\to X$ of the relative Frobenius endomorphism on $X\otimes_\Wnk k$.
	Then $d\Fr_X ^n = 0$.
	Thus the condition $da_1 = 0$ holds if we twist the given correspondence by the $n$-th power of the Frobenius lift $\Fr_X$.
\end{itemize}
\end{rmk}

Before starting the proof, we recall a notation of residue symbols.
We consider the following commutative diagram
\[
	\xymatrix{
		\bt \ar@{^{(}-{>}}[r]^i \ar[dr]_g & Z \ar[d]^f \\
		& S, \\
	}
\]
where $i$ is a closed immersion, $f$ is smooth of relative dimension $d$, and $g$ is \'etale and finite.

Since $g$ is \'etale, the natural morphism
\[
	\cI_\bt/\cI_\bt^2 \to \Omega_{Z/S}^1\ts_{\cO_Z}\cO_Z/\cI_\bt
\]
is isomorphic.
Putting $\bigwedge^d$ to the both sides of the above isomorphism, we have
\[
	\bigwedge^d \cI_\bt/\cI_\bt^2 \isomto \Omega_{Z/S}^d\ts_{\cO_Z}\cO_Z/\cI_\bt.
\]
We note that $\bigwedge^d \cI_\bt/\cI_\bt^2$ is an invertible sheaf on $\bt$.

Now we take an $\cO_Z$-sequence $s_1,\dots,s_d$ generating the defining ideal $\cI_\bt$ of $\bt$
and a global section $\omega$ of $\Omega_{Z/S}^d$.
We denote by $\overline{\omega}$ the global section of $\Omega_{Z/S}^d\ts_{\cO_Z}\cO_Z/\cI_\bt$ obtained from $\omega$
and by $\overline{s}_1, \ldots, \overline{s}_d$ the global sections of $\cI_\bt/\cI_\bt^2$ obtained from $s_1, \ldots, s_d$.
By the above isomorphism, we can regard $\overline{\omega}$ as a global section of $\bigwedge^d \cI_\bt/\cI_\bt^2$.
For any $\tau\in\Gam(\bt,\cO_\bt)$, we define
\begin{align*}
	\Res_{Z/S}\left( \tau \quad
		\begin{matrix}
			\omega \\ s_1\quad\cdots\quad s_d
		\end{matrix}
		\right)
	&:=
	\Tr_{\bt/S} \left( \tau \cdot \overline{\omega}\ts (d\overline{s}_1 \wedge\cdots\wedge d\overline{s}_d)^{-1} \right) ,
\end{align*}
which is an element of $\Gam(S,\cO_S)$.

\begin{proof}[Proof of Theorem \ref{thm31}]
We put $m = \dim X$.
We denote by $d_1$ (resp. $d_2$) the composite of $d:\cO_{X\times_S X}\to\Omega_{X\times_S X/S}$ and
the first (resp. second) projection of $\Omega_{X\times_S X/S} \isomto \pr_2^*\Omega_{X/S} \oplus \pr_1^*\Omega_{X/S}$.

By the corollary of Lefschetz-Verdier trace formula \cite[Exp.\ III, Thm.\ 6.10 and Rem.\ 6.11]{SGA5}, we obtain
\begin{align*}
	& \Tr(u_* \mid \RGa(X,\cK)) \\
	= & \sum_{\bt\in\pi_0(\Fix a)}\Res_{X\times_S X /S}\left( \Tr(u_\bt \mid \cK_\bt) \quad
		\begin{matrix}
			d_1 s_1 \wedge \cdots \wedge d_1 s_m \wedge d_2 t_1 \wedge \cdots \wedge d_2 t_m \\
			s_1 \quad \cdots \quad s_m \quad t_1 \quad \cdots \quad t_m \\
		\end{matrix}
		\right),
\end{align*}
where $s_1,\dots,s_m$ is an $\cO_{X\times_S X}$-sequence 
generating the defining ideal of $\Gam$ nearby $\bt$, and
$t_1,\dots,t_m$ is an $\cO_{X\times_S X}$-sequence defined as follows.

Since $a_2 \circ i_\bt$ is a closed immersion and $\bt$ is \'etale over $S$,
there are an open neighborhood $X'$ of $\bt$ in $X$ and an \'etale morphism $X'\to S[T_1,\dots,T_m]$
such that the following diagram is commutative:
\[
	\xymatrix{
		\bt \ar@{^{(}-{>}}[r] \ar[d] & X' \ar[d] \\
		S \ar@{^{(}-{>}}[r] & S[T_1,\dots,T_m] ,\\
	}
\]
where the bottom row is the zero section.
We denote by $x_1,\dots,x_m$ the sections of $\cO_X'$ corresponding to $T_1,\dots,T_m$.
By the above diagram, we obtain $a_2(x_i) \in \cI_\bt$ for each $i$.
In addition, $1\otimes x_i - x_i\otimes 1 , \ldots, 1\otimes x_m - x_m\otimes 1$ generate
the defining ideal of diagonal $\Delta_X\cln X\immto X\times_S X$ nearby $\bt$ since $X'$ is \'etale over $S[T_1,\dots,T_m]$.
We put $t_i = 1\otimes x_i - x_i\otimes 1$.

Then we obtain
\begin{align*}	
	  &	\Res\left( \Tr(u_\bt \mid \cK_\bt) \quad
		\begin{matrix}
			d_1 s_1 \wedge \cdots \wedge d_1 s_m \wedge d_2 t_1 \wedge \cdots \wedge d_2 t_m \\
			s_1 \quad \cdots \quad s_m \quad t_1 \quad \cdots \quad t_m \\
		\end{matrix}
		\right) \\
	= &	\Res\left( \Tr (u_\bt \mid \cK_\bt) \quad
		\begin{matrix}
			d_1 s_1 \wedge \cdots \wedge d_1 s_m \wedge d(1\otimes x_1) \wedge \cdots \wedge d(1\otimes x_m) \\
			s_1 \quad \cdots \quad s_m \quad t_1 \quad \cdots \quad t_m \\
		\end{matrix}
		\right) \\
	= & \Res\left( \Tr (u_\bt \mid \cK_\bt) \quad
		\begin{matrix}
			d s_1 \wedge \cdots \wedge d s_m \wedge d(1\otimes x_1) \wedge \cdots \wedge d(1\otimes x_m) \\
			s_1 \quad \cdots \quad s_m \quad t_1 \quad \cdots \quad t_m \\
		\end{matrix}
		\right) \\
	= & \Res\left( \Tr (u_\bt \mid \cK_\bt) \quad
		\begin{matrix}
			d(a_2(x_1)) \wedge \cdots \wedge d(a_2(x_m)) \\
			a_2(x_1)-a_1(x_1) \quad \cdots \quad a_2(x_m)-a_1(x_m)
		\end{matrix}
		\right) ,
\end{align*}
where we write $\Res$ instead of $\Res_{X\times_S X /S}$ for simplicity.
At the last equality we use the property of \cite[III.9\ (R3)]{Ha}.

By the definition of $x_1, \dots, x_m$, for each $i$ we can write 
\[
	a_2(x_i) = \sum_{j=1}^{m}c_{ij}(a_2(x_j)-a_1(x_j))
\]
nearby $\bt$, where each $c_{ij}$ is a section of $\cO_\Gam$.
Since $a_1 \circ i_\bt = a_2 \circ i_\bt$ and $da_1 = 0$ by the assumption, we have
\[
	da_2(x_i) = \sum_{j=1}^{m}c_{ij}da_2(x_j)
\]
as sections of $i_\bt^*\Omega_{\Gam/S}$.
Thus we obtain $c_{ij} = \delta_{ij}$ in $\cO_\bt$, where $\delta_{ij}$ is Kronecker's delta.
Therefore, by \cite[III.9\ (R6)]{Ha}, we have
\begin{align*}
	  & \Res\left( \Tr (u_\bt \mid \cK_\bt) \quad
		\begin{matrix}
			d(a_2(x_1)) \wedge \cdots \wedge d(a_2(x_m)) \\
			a_2(x_1)-a_1(x_1) \quad \cdots \quad a_2(x_m)-a_1(x_m)
		\end{matrix}
		\right) \\
	= & \Res\left( \Tr (u_\bt \mid \cK_\bt) \quad
		\begin{matrix}
			\det(c_{ij}) \cdot d(a_2(x_1)-a_1(x_1)) \wedge \cdots \wedge d(a_2(x_m)-a_1(x_m)) \\
			a_2(x_1)-a_1(x_1) \quad \cdots \quad a_2(x_m)-a_1(x_m)
		\end{matrix}
		\right) \\
	= & \Tr_{\bt/S}(\Tr(u_\bt \mid \cK_\bt) \det(c_{ij})|_{\bt}) \\
	= & \Tr_{\bt/S}(\Tr(u_\bt \mid \cK_\bt)).
\end{align*}

\end{proof}

\end{section}

\begin{section}{Some linear algebra}\label{sect5}
In this section, we recall some linear algebra for Section \ref{sect6} and \ref{sect7}.
We write $\sg_0\cln \WnFq\to\WnFq$ (resp. $\sg\cln\Wnk\to\Wnk$)
for an endomorphism induced by the $p$-th power map on $\Fq$ (resp. on $k$).
We put $d = [\Fq:\Fp]$.

\begin{lem}\label{lem311}
Let $M_0$ be a $\WnFq$-module of finite type
and $\vP_0\cln M_0 \to M_0$ a $\sg_0$-linear endomorphism.
We put $F_0 = \vP_0^d$ and $M=M_0\ts_\WnFq\Wnk$.
We write $\vP\cln M \to M$ (resp. $F\cln M \to M$)
for the $\sg$-linear (resp. $\Wnk$-linear) extension of $\vP_0$ (resp. $F_0$) to $M$.
Then the following assertions hold:
\begin{itemize}
\item[\rm{(1)}]
	The morphism $1-\vP\cln M \to M$ is surjective.
\item[\rm{(2)}]
	We have $M^{1-\vP}/p\left(M^{1-\vP} \right) = \left( M/pM \right)^{1-\vP}$.
\item[\rm{(3)}]
	The canonical morphism $M^{1-\vP}\ts_\Zoverpn\Wnk \to M$ is injective.
	Moreover, $\vP$ and $F$ is nilpotent on the cokernel of this morphism.
\end{itemize}
\end{lem}

\begin{proof}
(1)
If $n=1$, this assertion follows from \cite[Exp.\ XXII, Prop.\ 1.2]{SGA7}.

We assume that the assertion holds for $n-1$.
We have the commutative diagram of the exact sequences
\[
\xymatrix
{
	0 \ar[r] & pM \ar[r] \ar[d]^{1-\vP} & M \ar[r] \ar[d]^{1-\vP} & M/pM \ar[r] \ar[d]^{1-\vP} & 0 \\
	0 \ar[r] & pM \ar[r] & M \ar[r] & M/pM \ar[r] & 0.
}
\]
We note that $pM$ (resp. $M/pM$) has a natural $W_{n-1}(k)$-module (resp. $k$-vector space) structure of finite type.
By the assumption of the induction (resp. the case $n=1$), the endomorphism $1-\vP$ on $pM$ (resp. $M/pM$) is surjective.
Hence, by the above exact sequence, $1-\vP$ on $M$ is surjective .

(2)
Applying the snake lemma to the above diagram, we have the exact sequence
\[
	0 \to (pM)^{1-\vP} \to M^{1-\vP} \to (M/pM)^{1-\vP} \to 0.
\]
Hence we have $M^{1-\vP}/(pM)^{1-\vP} = (M/pM)^{1-\vP}$.

On the other hand, we have the commutative diagram of the exact sequences
\[
\xymatrix{
	0 \ar[r] & M[p] \ar[r] \ar[d]^{1-\vP} & M \ar[r]^{\times p} \ar[d]^{1-\vP} & pM \ar[r] \ar[d]^{1-\vP} & 0 \\
	0 \ar[r] & M[p] \ar[r] & M \ar[r]^{\times p} & pM \ar[r] & 0,
}
\]
where we put $M[p] = \Ker(M\xrightarrow{\times p}M)$.
We note that $M[p]$ has a natural $k$-vector space structure of finite type.
As we have shown above, the endomorphism $1-\vP$ on $M[p]$ is surjective.
Applying the snake lemma to this diagram, we have $p(M^{1-\vP}) = (pM)^{1-\vP}$.

Therefore we obtain $M^{1-\vP}/p(M^{1-\vP}) = (M/pM)^{1-\vP}$.

(3)
If $n=1$, the assertion follows from \cite[Exp.\ XXII, Cor.\ 1.1.10]{SGA7}.
We assume that the assertion holds for $n-1$.

By the proved assertion (2), we have the exact sequence
\[
\xymatrix{
	0 \ar[r] & (pM)^{1-\vP} \ar[r] & M^{1-\vP} \ar[r] & (M/pM)^{1-\vP} \ar[r] & 0.
}
\]
Tensoring this diagram with $\Wnk$, we obtain the exact sequence
\[
\xymatrixcolsep{1pc}\xymatrix{
	0 \ar[r] & (pM)^{1-\vP}\ts W_{n-1}(k) \ar[r] & M^{1-\vP}\ts\Wnk \ar[r] & (M/pM)^{1-\vP}\ts k \ar[r] & 0,
}
\]
where we write
\[
	\ts W_{n-1}(k),\quad \ts\Wnk \quad\text{and}\quad \ts k
\]
for
\[
	\ts_{\Z/p^{n-1}}W_{n-1}(k),\quad \ts_{\Z/p^n}\Wnk \quad\text{and}\quad \ts_\Fp k.
\]
We have the commutative diagram of the exact sequences
\[
\xymatrixcolsep{1pc}\xymatrix{
	0 \ar[r] & (pM)^{1-\vP}\ts W_{n-1}(k) \ar[r] \ar[d] & M^{1-\vP}\ts\Wnk \ar[r] \ar[d] & (M/pM)^{1-\vP}\ts k \ar[r] \ar[d] & 0 \\
	0 \ar[r] & pM \ar[r] & M \ar[r] & M/pM \ar[r] & 0 \lefteqn{.}
}
\]
By the assumption of the induction (resp. the case $n=1$), 
the left (resp. right) vertical arrow of the above diagram is injective.
Hence the morphism
\[
	M^{1-\vP}\ts\Wnk \to M
\]
is injective.

We prove the latter statement.
We put
\[
	M' = \Coker(M^{1-\vP}\ts\Wnk \to M).
\]
Applying the snake lemma to the above diagram, we have
\[
	M'/pM' \simeq \Coker\left( (M/pM)^{1-\vP}\ts k \to M/pM \right).
\]
We need the following lemma:
\begin{lem}\label{lem312}
Let $M'$ be a $\Wnk$-module
and $\vP\cln M' \to M'$ be a $\Z/p^n$-linear endomorphism.
We assume that $\vP$ is nilpotent on $M'/pM'$.
Then $\vP$ is nilpotent on $M'$.
\end{lem}
\begin{proof}
For the assumption, there exists an integer $e \geq 1$ such that $\vP^e = 0$ on $M'/pM'$.
Let $x_0$ be an element of $M'$.
For each integer $i\geq 0$, we can inductively find $x_{i+1}\in M'$ such that $\vP^e(x_i) = px_{i+1}$.
Thus we have
$\vP^{ne}(x)=p^nx_n=0.$
Therefore $\vP$ is nilpotent.
\end{proof}
We return to the proof of Lemma \ref{lem311}.(3).
By the case $n=1$, the endomorphism $\vP$ on $M'/pM'$ is nilpotent.
Therefore, by Lemma \ref{lem312}, $\vP$ is also nilpotent on $M'$.

We will show that $F$ is nilpotent on $M'$.
We fix an integer $e$ such that $\vP^{de} = 0$ on $M'/pM'$.
Let $x_0$ be an element of $M_0$ and $\al$ an element of $\Wnk$.
We denote by $x_0\ts\al$ the class of $x_0\ts\al\in M$ in $M'$.
By Lemma \ref{lem312}, it suffices to show that the element
\[
	F^e(x_0\ts\al) = \al(\vP_0^{de}(x_0)\ts1)
\]
belongs to $pM'$.
By the definition of $e$, the element 
\[
	\vP^{de}(x_0\ts\al) = \sg^{de}(\al)(\vP_0^{de}(x_0)\ts 1)
\]
belongs to $pM'$.
Hence we have either $\sg^{de}(\al) \in p\Wnk$ or $\vP_0^{de}(x_0)\ts 1 \in p\Wnk$.
For the former case, we have $\al\in p\Wnk$.
Therefore, in both cases, we have $F^e(x_o\ts\al) \in pM'$.
\end{proof}

\begin{lem}\label{lem313}
We use the notation of Lemma \ref{lem311}.
We assume that $M$ is free over $\Wnk$ and $M^{1-\vP}$ is free over $\Z/p^n$.
Then there exists an integer $N \geq 1$ such that the following assertion holds:
Let $\vp$ be a $\Wnk$-linear endomorphism on $M$ 
which is defined over $\WnFq$ and satisfies $\vp\circ\vP = \vP\circ\vp$.
For any $m \geq N$, we have
\[
	\Tr(\vp\circ F^m \mid M^{1-\vP}) = \Tr(\vp\circ F^m \mid M).
\]
If $n=1$, then $N=1$ satisfies the above condition.
\end{lem}

\begin{rmk}\label{rmk314}
By the assumption of $\vp$,
we have $\vp\circ F = F\circ\vp$ and
$M^{1-\vP}$ is $\vp\circ F^m$-stable for any $m$.
\end{rmk}

\begin{proof}[Proof of Lemma \ref{lem313}]
By Lemma \ref{lem311}(3), we have the exact sequence
\[
	0 \to M^{1-\vP}\ts_{\Z/p^n}\Wnk \to M \to M' \to 0 ,
\]
where we put $M' = \Coker(M^{1-\vP}\ts_{\Z/p^n}\Wnk \to M)$.
By the assumption, $M^{1-\vP}\ts\Wnk$ is free $\Wnk$-module of finite rank.
Since $\Wnk$ is a Gorenstein ring, $M^{1-\vP}\ts\Wnk$ is an injective $\Wnk$-module.
Hence the above sequence splits.
Thus we have
\[
	\Tr(\vp\circ F^m \mid M) = \Tr(\vp\circ F^m \mid M^{1-\vP}) + \Tr(\vp\circ F^m \mid M').
\]

Let $N$ be the minimal number of the integers $N'$ such that the equation $\vP^{dN'}=0$ on $M'$ holds.
By Lemma \ref{lem311}(3), this value is well-defined and we have $F^m = 0$ on $M'$ for any $m \geq N'$.
Therefore we have $\Tr(\vp\circ F^m \mid M') = 0$ and the equation in Lemma \ref{lem313} holds.

We assume that $n=1$.
By Remark \ref{rmk314}, $\vp\circ F$ is nilpotent on $M'$.
Hence all of the eigenvalues of $\vp\circ F$ are zero.
Therefore, for any $m\geq 1$, we have $\Tr(\vp\circ F^m \mid M') = 0$.
\end{proof}
\end{section}

\begin{section}{Main theorem for the case $n=1$}\label{sect6}

In this section we prove Theorem \ref{thm12} under the condition (ii).

\begin{thm}\label{thm41}
Let $\Xob$ and $\Gamob$ be proper smooth $\Fq$-schemes
and $X_0$ an open $\Fq$-subscheme of $\Xob$ such that the complement $\Xb \smin X$ is associated to a Cartier divisor.
Let $\ab_0 \cln \Gamob \to \Xob \times _\Fq \Xob$ be an $\Fq$-morphism.
We write $a_0 \cln \Gam_0 \to X_0\times_k X_0$ for the morphism such that the following diagram is cartesian:
\[
\xymatrix{
	\Gam_0 \ar[r]^-{a_0} \ar@{^{(}-{>}}[d] \ar@{}[rd]|-{\square} & X_0\times_\Fq X_0 \ar@{^{(}-{>}}[d] \\
	\Gamob \ar[r]^-{\ab_0} & \Xob \times_\Fq \Xob .\\
}
\]
Suppose that $a_1$ is proper, $\ab_2$ is \'etale, and $\ab$ is a closed immersion.
Then, for any integer $m \geq 1$, any smooth constructible $\Fp$-sheaf $\Gob$ on $X_0$, and any $u_0 \in \Hom(a_{01}^*G_0, a_{02}^*G_0)$, we obtain
\begin{equation*}
	\Tr ((u\circ \Fr^m)_! \mid \RGac (X, G)) = \sum_{z \in \Fix a^{(m)}} \Tr((u\circ \Fr^m)_z \mid G_z),
\end{equation*}
where we put $G_0 = \Gob|X_0$.
\end{thm}

\begin{rmk}\label{rmk42}
\begin{itemize}
\item[(a)]
Since $a_{01}$ is proper, the following diagram is cartesian as that of topological spaces:
\[
\xymatrix{
	\Gam \ar[r]^{a_{01}} \ar@{^{(}-{>}}[d]_{j_\Gam} & X \ar@{^{(}-{>}}[d]^{j_X} \\
	\Gamb \ar[r]^{\ab_{01}} & \Xb . \\
}
\]
Thus we have $\ab_{02}^{\,-1}(\Xob \smin X_0) \subset \ab_{01}^{\,-1}(\Xob \smin X_0)$ as topological spaces.
\item[(b)]
Since $\Gob$ is smooth and $X_0$ is dense in $\Xob$, 
there exists a unique morphism $\ub_0\cln\ab_{01}^*\Gob\to\ab_{02}^*\Gob$ such that ${\ub_0}|_{\Gam_0} = u_0$
(see for example \cite[Exp.\ XVI, Prop.\ 3.2]{SGA4}).
\end{itemize}
\end{rmk}

\begin{proof}[Proof of Theorem \ref{thm41}]
By changing the field of definition of $X$, we may assume that $m=1$.

We put $\cGob = \Gob\otimes_{\Fp}\cO_{\Xob}$ and $\cGo = \cGob|_{X_0} = G_0\otimes_{\Fp}\cO_{X_0}$.
We denote by $\cI_0$ the reduced defining ideal of $\Xob \smin X_0$.
We put $\cG_0' = \cI_0\cGob$,
which is a locally free $\cO_{\Xob}$-module since $\cI$ is locally generated by a regular section.
We note that we have $\cG_0' = \Gob \otimes_{\Fp} \cI_0$ since $\Gob$ is a smooth $\Fq$-sheaf.

If $Y$ is a scheme of characteristic $p$, we define the morphism $f_Y$ of schemes by $f_Y = (\id_Y, \vP_{\cO_Y})$,
where $\vP_{\cO_Y}\cln\cO_Y \to \cO_Y$ is the $p$-th power map.
For a $p$-linear morphism of $\cO_Y$-module sheaves
$\vP \cln \cG_1 \to \cG_2$,
we denote by $\vP'\cln f_Y^*\cG_1 \to \cG_2$ the $\cO_Y$-linear homomorphism associated to $\vP$.

We put $\vP_\cGob = \id_\Gob \otimes \vP_{\cO_{\Xob}}$ and we also define $\vP_{\cG_0}$, $\vP_\cGb$, and $\vP_\cG$ as well.
By abuse of notation, we sometimes denote by $\vP$ these $p$-linear homomorphisms of sheaves.
We put $F_0 = \id_{\Gob}\ts \vP'_{\cO_{\Xob}} \cln f_{\Xob}^*\cGob \to \cGob$.

Then we have
\[
	F_0(\cI_0^n\cGob) \subset \cI_0^{np}\cGob.
\]
Thus the morphism $F_0^d\cln \Fr_{\Xob}^*\cGob \to \cGob$ is factorized as
\begin{equation}\label{dgm00}
\begin{split}
\xymatrix{
	\Fr_{\Xob}^*\cGo' \ar@{-->}[r] \ar[d] & \cI_0^q\cGob \ar@{^(->}[r] & \cGo'  \ar@{^(->}[d] \\
	\Fr_{\Xob}^*\cGob \ar[rr]^{F_0^d} & & \cGob,
}
\end{split}
\end{equation}
where we put $d = [\Fq:\Fp]$.
Changing the base of the above diagram by $\Fq\to k$, we obtain the diagram
\begin{equation}\label{dgm01}
\begin{split}
\xymatrix{
	\Fr_X^*\cG' \ar[r] \ar[d] & \cI^q\cGb \ar@{^(->}[r] & \cG' \ar@{^(->}[d] \\
	\Fr_X^*\cGb \ar[rr] & & \cGb .
}
\end{split}
\end{equation}
We write $\Fr_{\cGb}$ for the bottom horizontal morphism and $\Fr_{\cG'}$ for the composition of the top horizontal morphisms.
For simplicity, we sometimes write $\Fr$ for these morphisms and their pull-backs.

We have 
\[
	{(\ab_{01}^{\,-1}\cI_0\cdot\cO_{\Gamob})}_\red \subset {(\ab_{02}^{\,-1}\cI_0\cdot\cO_{\Gamob})}_\red
\]
by Remark \ref{rmk42} (a).
Since $\ab_{02}$ is \'etale, we obtain
\[
	\ab_{01}^{\,-1}\cI_0\cdot\cO_{\Gamob} \subset {(\ab_{01}^{\,-1}\cI_0\cdot\cO_{\Gamob})}_\red 
		\subset {(\ab_{02}^{\,-1}\cI_0\cdot\cO_{\Gamob})}_\red = \ab_{02}^{\,-1}\cI_0\cdot\cO_{\Gamob}.
\]
Hence, by the flatness of $\ab_{02}$,
the morphism $\ub_0 \otimes \id\cln \ab_{01}^*\cGb \to \ab_{02}^*\cGb$ is factorized as
\begin{equation}\label{dgm1}
\begin{split}
\xymatrix{
	\ab_{01}^*\cGo' \ar@{-->}[r]^{u_0'} \ar[d] & \ab_{02}^*\cGo' \ar@{^{(}-{>}}[d]\\
	\ab_{01}^*\cGob \ar[r]^{\ub_0 \ts \id} & \ab_{02}^*\cGob \\
}
\end{split}
\end{equation}
and by the diagram (\ref{dgm00}) we have
\begin{equation}\label{eqn02}
	(u_0'\circ\Fr)(\ab_{01}^{(1)*}\cG') \subset \ab_{02}^*\cI_0^q\cGob \subset \ab_{02}^*\cI_0\cGo'.
\end{equation}
Thus we have the diagram
\begin{equation}\label{dgm2}
\begin{split}
\xymatrix{
	\RGac(X,G) \ar[r]^-{\RGa(\adj)} \ar[d]
		& \RGa(\Gamb,\aib^*j_!G) \ar[r]^{\RGa(j_!u)} \ar[d]
		& \RGa(\Gamb,\azb^*j_!G) \ar[r]^{\RGa(\Tr)} \ar[d]
		& \RGac(X,G) \ar[d] \\
	\RGa(\Xb,\cG') \ar[r]^-{\RGa(\adj)}     
		& \RGa(\Gamb,\aib^*\cG') \ar[r]^{\RGa(u')} 
		& \RGa(\Gamb,\azb^*\cG') \ar[r]^{\RGa(\Tr)}     
		& \RGa(\Xb,\cG') \lefteqn{.}
}
\end{split}
\end{equation}
We note that the composite of the top horizontal morphisms in the diagram (\ref{dgm2}) coincides with $u_!$
and that of the bottom horizontal morphisms is defined over $\Fq$.

We check the commutativity of the diagram (\ref{dgm2}).
The commutativity of the left and right squares follows from Proposition \ref{prop21}.
To show that the middle is commutative, it suffices to check the commutativity of the square
\begin{equation}\label{dgm21}
\begin{split}
\xymatrix{
	\aib^*j_!G \ar[r]^{j_!u} \ar[d] & \azb^*j_!G \ar[d] \\
	\aib^*\cG' \ar[r]^{u'} & \azb^*\cG'.
}
\end{split}
\end{equation}
In the following diagram
\[
\xymatrix{
	 & \aib^*j_!G \ar[ld] \ar[dd]|{\hole} \ar[rr] & & \azb^*j_!G \ar[ld] \ar[dd] \\
	\aib^*\Gb \ar[dd] \ar[rr] & & \azb^*\Gb \ar[dd] & \\
	 & \aib^*\cG' \ar[ld] \ar[rr]|(0.54){\hole} & & \azb^*\cG' \ar[ld]^{(*)} \\
	\aib^*\cGb \ar[rr] & & \azb^*\cGb ,& 
}
\]
all the squares in the above diagram except (\ref{dgm21}) are commutative
and the morphism $(*)$ is injective. 
Thus the commutativity of the diagram (\ref{dgm21}) follows.

Also we have the commutative diagram
\begin{equation}\label{dgm3}
\begin{split}
\xymatrix{
	\RGa(\Xb,\cG') \ar[r]^-{\RGa(\adj)} \ar[d]_{\RGa(\vP)}    
		& \RGa(\Gamb,\aib^*\cG') \ar[r]^{\RGa(u')} \ar[d]_{\RGa(\vP)}
		& \RGa(\Gamb,\azb^*\cG') \ar[r]^{\RGa(\Tr)} \ar[d]_{\RGa(\vP)}
		& \RGa(\Xb,\cG') \ar[d]_{\RGa(\vP)} \\
	\RGa(\Xb,\cG') \ar[r]^-{\RGa(\adj)}
		& \RGa(\Gamb,\aib^*\cG') \ar[r]^{\RGa(u')}
		& \RGa(\Gamb,\azb^*\cG') \ar[r]^{\RGa(\Tr)}
		& \RGa(\Xb,\cG') .
}
\end{split}
\end{equation}
The commutativity of the right and left squares follows from the commutative squares
\[
\xymatrix{
	f_{\Xob}^*\cGo' \ar[r]^-{\adj} \ar[d]_{\vP'}
		& {\aoib}_*\aoib^*f_{\Xob}^*\cGo' \ar[d]_{\vP'} \\
	\cGo' \ar[r]^-{\adj} & {\aoib}_*\aoib^*\cGo' \\
}
\]
and
\[
\xymatrix{
	{\aozb}_*\aozb^*f_{\Xob}^*\cGo' \ar[r]^-{\Tr} \ar[d]_{\vP'} & f_{\Xob}^*\cGo' \ar[d]_{\vP'} \\
	{\aozb}_*\aozb^*\cGo' \ar[r]^-{\Tr} & \cGo' \lefteqn{.}
}
\]
The commutativity of the middle of the diagram (\ref{dgm3}) follows from the commutative square
\[
\xymatrix{
	\aib^*\cGb \ar[r]^-{\ub \otimes \id} \ar[d]_{\vP} & \azb^*\cGb \ar[d]_{\vP}\\
	\aib^*\cGb \ar[r]^-{\ub \otimes \id} & \azb^*\cGb
}
\]
and the commutative diagram (\ref{dgm1}).

By \cite[Fonction $L$ mod.\ $\ell^n$, Lem.\ 3.3]{SGA4h}, we obtain the exact sequence
\begin{equation}\label{seq1}
\xymatrix{
	0 \ar[r] & j_!G \ar[r] & \cG' \ar[r]^{1-\vP} & \cG' \ar[r] & 0.
}
\end{equation}
which induces the following long exact sequence
\[
\xymatrixcolsep{3pc}\xymatrix{
	\cdots \ar[r] & H_c^i(X,G) \ar[r] & H^i(\Xb,\cG') \ar[r]^{1 - H^i(\vP)} & H^i(\Xb,\cG') \ar[r] & \cdots.
}
\]
Since $\Xb$ is proper over $k$ and $\cG'$ is a coherent $\cO_{\Xb}$-module sheaf,
the $k$-vector space $H^i(\Xb,\cG')$ is finite-dimensional for each $i$.
Hence by Lemma \ref{lem311}(1) the $p$-linear map $1 - H^i(\vP)$ is surjective.
Thus, for each index $i$, we have the exact sequence
\[
\xymatrixcolsep{3pc}\xymatrix{
	0 \ar[r] & H_c^i(X,G) \ar[r] & H^i(\Xb,\cG') \ar[r]^{1 - H^i(\vP)} & H^i(\Xb,\cG') \ar[r] & 0.
}
\]
By the diagram (\ref{dgm3}), we have $u' \circ \vP = \vP \circ u'$.
Therefore, by the diagram (\ref{dgm2}) and Lemma \ref{lem313}, we have
\begin{align}\label{eqn41}
\begin{split}
	\Tr((u \circ \Fr)_! \mid H_c^i(X,G)) &= \Tr(u_! \circ (\Fr_G)_! \mid H_c^i(X,G)) \\
		& = \Tr(u'_* \circ (\Fr_{\cG'})_* \mid H^i(\Xb,\cG')) \\
		&= \Tr((u' \circ \Fr)_* \mid H^i(\Xb,\cG')) .
\end{split}
\end{align}

On the other hand, since $d\Fr_{\Xb} = 0$, $\azb$ is \'etale and $\cG'$ is locally free of finite rank,
we can apply Theorem \ref{thm31} to $\Tr((u'\circ\Fr)_*) \mid \RGa(\Xb,\cG'))$:
\begin{equation}\label{eqn42}
	\Tr((u'\circ \Fr)_* \mid \RGa(\Xb,\cG')) = \sum_{z\in \Fix\ab^{(1)}}\Tr((u'\circ\Fr)_z \mid \cG'_z).
\end{equation}
We note that, since in this case the base scheme is $\Spec k$,
we can identify connected components of $\Fix \ab^{(1)}$ with closed points of $\Fix \ab^{(1)}$
by associating $\bt$ to the underlying space $z$ of $\bt$.\\
\begin{itemize}
\item
	If $z$ lies in $\Fix a^{(1)}$, then we have
	\[
		\Tr((u'\circ\Fr)_z \mid \cG'_z) = \Tr((u\circ\Fr)_z \mid G_z)
	\]
	by the definition of $u'$.
\item
	If $z$ does not belong to $\Fix a^{(1)}$, then $\Fr_\Gamb\circ\aib(z) = \azb(z) \in \Xb \smin X$ by Remark \ref{rmk42}(a).
	By the inclusion (\ref{eqn02}), $u'\circ\Fr$ is factorized as follows:
	\[
	\xymatrix{
		 & \azb^*\cI\cG' \ar@{^{(}->}[d]^{(**)} \\
		\aib^{(1)*}\cG' \ar@{-->}[ur] \ar[r]_{u'\circ\Fr} & \azb^*\cG'.
	}
	\]
	In the above diagram, the pull-back of $(**)$ by $i_z$ is zero.
	Thus we have $(u'\circ\Fr)_z = 0$ and $\Tr((u'\circ\Fr)_z \mid \cG'_z) = 0$.
\end{itemize}
Therefore, from (\ref{eqn41}), (\ref{eqn42}) and the above calculation of the local terms, we obtain Theorem \ref{thm41}.
\end{proof}

\end{section}

\begin{section}{Main theorem for the proper case}\label{sect7}

In this section we prove Theorem \ref{thm12} under the condition (iii).
We write $\sg_0$ (resp. $\sg$) for the Frobenius map on $\WnFq$ (resp. $\Wnk$).

\begin{thm}\label{thm51}
Let $X$ and $\Gam$ be proper smooth $k$-schemes defined over $\Fq$.
Let $a\cln \Gam\immto X\times X$ be a closed immersion defined over $\Fq$ such that $a_2$ is \'etale.
Let $K_0$ be a complex in $\Dperf(X_0,\Zoverpn)$.
We assume the following conditions:
\begin{itemize}
\item
	There exists a triplet of lifts $(\sX_0,\tGa_0,\ta_0)$ of $(X_0,\Gam_0,a_0)$ to $\WnFq$
	such that $\sX$ and $\tGa$ are smooth over $\Wnk$, $\ta$ is closed immersion, and $\ta_2$ is \'etale.
\item
	There exists a $\sg_0$-linear endomorphism $\vP_{\sX_0}\cln \cO_{\sX_0}\to\cO_{\sX_0}$
	which is a lift of the $p$-th power map $\vP_{X_0}\cln \cO_{X_0}\to\cO_{X_0}$.
\item
	$H^i(X,K)$ (resp. $H^i(\sX,K\ts_{\Z/p^n}\cO_\sX)$) is free over $\Z/p^n$ (resp. $\Wnk$) for each $i$.
\end{itemize}
Then there exists an integer $N$ such that,
for any $m \geq N$ and any $u_0 \in \Hom(a_{01}^*K_0,a_{02}^*K_0)$, we obtain
\[
	\Tr((u\circ \Fr^m)_* \mid \RGa (X, K)) = \sum_{z \in \Fix a^{(m)}} \Tr((u\circ\Fr^m)_z \mid K_z).
\]
\end{thm}

\begin{proof}	
We put $\sS = \Spec\Wnk$.
Since the morphism $\Spec k \to \sS$ is surjective, radiciel and finite,
the pull-back functor from the \'etale topos of $\sX$ (resp. $\tGa$) 
to the \'etale topos of $X$ (resp. $\Gam$) is an equivalence of categories
and $D(\sS,\Zoverpn)$ is naturally identified with $D(\Zoverpn)$.
Let $\tK_0$ be the complex on $\sX_0$ such that ${\tK_0}|_{X_0} = K_0$
and $\tu_0\cln \ta_{01}^*\tK_0 \to \ta_{02}^*\tK_0$ the morphism such that ${\tu_0}|_{\Gam_0} = u_0$.
Then we have
\[
	\Tr((u\circ \Fr^m)_* \mid \RGa(X,K)) = \Tr((\tu\circ \Fr^m)_* \mid \RGa(\sX,\tK))
\]
and
\[
	\Tr((u\circ\Fr^m)_z \mid K_z) = \Tr({(\tu\circ\Fr^m)}_{\tz} \mid {\tK}_{\tz}),
\]
for any $z \in \Fix a^{(m)}$, where $\tz$ is the element of $\Fix \ta^{(m)}$ corresponding to $z$.
Hence it suffices to show the formula of Theorem \ref{thm51} for $(\sX,\tGa,\ta)$.

For simplicity, we replace the symbols $(\sS,\sX,\tGa,\ta,\tu)$ with $(S,X,\Gam,a,u)$.

By the existence of $\vP_{X_0}$, we have the short exact sequence
\[
\xymatrix{
	0 \ar[r] & \Z/p^n \ar[r] & \cO_X \ar[r]^{1-\vP} & \cO_X \ar[r] & 0 \lefteqn{.}
}
\]
We put $\cK = K\otimes_\Zoverpn\cO_X$.
Tensoring this diagram with $K$ over $\Z/p^n$ and acting $\RGa(X,-)$, we have the distinguished triangle
\[
\xymatrix{
	\RGa(X,K) \ar[r] & \RGa(X,\cK) \ar[r]^{1-\vP} & \RGa(X,\cK) \ar[r] & \RGa(X,K)[1] \lefteqn{.}
}
\]
By Lemma \ref{lem311} (1), we have the short exact sequence
\[
\xymatrix{
	0 \ar[r] & H^i(X,K) \ar[r] & H^i(X,\cK) \ar[r]^{1-H^i(\vP)} & H^i(X,\cK) \ar[r] & 0 \lefteqn{.}
}
\]

We put $f_{X_0} = (\id, \vP_{X_0})\cln (X_0,\cO_{X_0}) \to (X_0, \cO_{X_0}).$
Let $\vP'\cln f_{X_0}^*\cO_{X_0} \to \cO_{X_0}$ be the $\cO_{X_0}$-linear homomorphism associated to $\vP_{X_0}$.
We write $\Fr_{\cO_X}$ for the pull back of $\vP'^d$ by $\WnFq \to \Wnk$
and define $\Fr_\cK = \Fr_K \ts \Fr_{\cO_X}$.
We put $u'=u\ts\id_{\cO_\Gam}:a_1^*\cK\to a_2^*\cK$.
By Lemma \ref{lem313}, we can take an integer $N$ such that for any $m \geq N$ we have
\begin{align*}
	\Tr((u\circ\Fr^m)_* \mid H^i(X,K))
	&= \Tr(u_*\circ (\Fr_K)_*^m \mid H^i(X,K)) \\
	&= \Tr(u'_* \circ (\Fr_\cK)_*^m \mid H^i(X,\cK)) \\
	&= \Tr((u'\circ\Fr^m)_* \mid H^i(X,\cK)).
\end{align*}
Taking the alternating sum and applying Theorem \ref{thm31}, we obtain
\begin{align*}
	\Tr((u\circ\Fr^m)_* \mid \RGa(X,K))
		&= \Tr((u'\circ\Fr^m)_* \mid \RGa(X,\cK)) \\
		&= \sum_{\bt\in\pi_0(\Fix a^{(m)})} \Tr((u'\circ\Fr^m)_\bt \mid \cK_\bt) \\
		&= \sum_{z\in\Fix a^{(m)}} \Tr((u\circ\Fr^m)_z \mid K_z) \lefteqn{.}
\end{align*}
Therefore Theorem \ref{thm51} holds.
\end{proof}
\end{section}

\begin{section}{Examples}\label{sect8}
We will see three examples.
At first, we see an easy example that we can show Theorem \ref{thm12} by direct calculation.
Secondly, we apply this theorem to elliptic curves over a finite field.
Finally, we see that this theorem would be false if coefficients are $p$-adic.

\begin{exam}
We see an example of the affine line over a finite field.
In the notation of Theorem \ref{thm41}, we put
\begin{itemize}\label{exam81}
\item
	$m=1$,
\item
	$X_0=\AFp$, $\Xob=\PFp$,
\item
	$a_0=(g_0,\id)\cln \AFp \to \AFp\times_{\Fq}\AFp$,
	where $g_0\cln\AFp\to\AFp$ is defined by $x\mapsto x+1$,
\item
	$\ab_0=(\gb_0,\id)$, where $\gb_0\cln\PFp\to\PFp$ is a natural extension of $g_0$,
\item
	$\Gob = \Z/p$, $\ub_0=\id$.
\end{itemize}

We have $H_c^i(\Ak,\Z/p) = 0$ for each $i$.
Hence the left-hand side of the formula in Theorem \ref{thm41} is zero.

On the other hand, the right-hand side coincides with 
the number of fixed points of $g\circ\Fr$,
that is, the number of roots of the equation $x^q+1=x$ in $k$.
It exactly equals $q$, hence it is zero modulo $p$.
\end{exam}

\begin{exam}
We show an easy application of Theorem \ref{thm41}.
Let $E_0$ be an elliptic curve over $\Fq$ and $O$ the origin of $E_0$.
Let $g_0$ be the isogeny on $E_0$.
In the notation of Theorem \ref{thm41}, we put
\begin{itemize}
\item
	$m=1$,
\item
	$\Xob=E_0$, $ X_0=E_0\smin \{O\}$,
\item
	$\Gam_0 = g_0^{-1}(X_0)$, $j_0 \cln \Gam_0 \immto X_0$, $g_0' = g_0|_{\Gam_0}$,
\item
	$\ab_0 = (g_0,\id)$, $a_0 = (g_0',j_0)$,
\item
	$\Gob = \Z/p$, $\ub_0 = \id$.
\end{itemize}
Then the above data satisfy the condition of Theorem \ref{thm41}.
The cohomology groups with proper supports of $X$ are as follows:
\[
	H_c^i(X,\Z/p) = 
	\begin{cases}
		0 & i=0 \\
		H^i(E,\Z/p) & \text{otherwise}.
	\end{cases}
\]
By using the Artin-Schreier sequence, we can compute $H^i(E,\Z/p)=0$ for any integer $i\geq 2$.
On the other hand, we can show that
\[
	\Fix a^{(1)} = \Fix(\Fr_E \circ g) \smin \{O\}.
\]
Thus, by applying Theorem \ref{thm41}, we have
\[
	- \Tr((\Fr_E\circ g)^* \mid H^1(E,\Z/p)) \equiv \#\Fix(\Fr_E \circ g) - 1 \pmod{p}.
\]
We fix a prime number $\ell \neq p$.
Using the Lefschetz trace formula for $\ell$-adic \'etale cohomology, we can show that
\[
	- \Tr((\Fr_E\circ g)^* \mid H^1(E,\Ql)) \equiv \#\Fix(\Fr_E \circ g) - 1 \pmod{p}.
\]
Note that $\Tr((\Fr_E\circ g)^* \mid H^1(E,\Ql))$ is an integer.
Thus we obtain
\[
	\Tr((\Fr_E\circ g)^* \mid H^1(E,\Z/p)) \equiv \Tr((\Fr_E\circ g)^* \mid H^1(E,\Ql)) \pmod{p}.
\]

If $E_0$ is supersingular, the Artin-Schreier sequence shows that $H^1(E,\Z/p)=0$.
Thus we have
\[
	\#\Fix(\Fr_E \circ g) \equiv 1 \pmod{p}
\]
and
\[
	\Tr((\Fr_E\circ g)^* \mid H^1(E,\Ql)) \equiv 0 \pmod{p}.
\]
\end{exam}

\begin{exam}
We remark that Deligne's conjecture seems to be false if coefficients are $p$-adic.
For example, in Example \ref{exam81}, we take $\Gob = \Z_p$.
We note that
\[
	H_c^i(\Ak,\Z/p^n) = 0
\]
for each $n$ and each $i$.
We define $H_c^i(\Ak,\Z_p):=\varprojlim_{n}H_c^i(\Ak,\Z/p^n)$.
Then we have $H_c^i(\Ak,\Z_p)=0$.
Hence the equation
\[
	\sum_i (-1)^i\Tr((\Fr^m)^* \mid H_c^i(\Ak,\Z_p)) = 0
\]
holds for each integer $m \geq 1$. However, we have
\[
	\# \{ x\in\Ak \mid x^{q^m} + 1 = x \} = q^m,
\]
which is not zero in $\Z_p$.
\end{exam}
\end{section}

\end{document}